\documentclass[12pt]{article}%
\usepackage{amsmath}
\usepackage{amsfonts}
\usepackage{amssymb}
\usepackage{graphicx}%
\providecommand{\U}[1]{\protect\rule{.1in}{.1in}}
\newtheorem{theorem}{Theorem}
\newtheorem{acknowledgement}[theorem]{Acknowledgement}

\newtheorem{corollary}[theorem]{Corollary}

\newtheorem{example}[theorem]{Example}

\newtheorem{lemma}[theorem]{Lemma}

\newenvironment{proof}[1][Proof]{\noindent\textbf{#1.} }{\ \rule{0.5em}{0.5em}}
\begin{document}

\title{Cycle products and efficient vectors in reciprocal matrices}
\author{Susana Furtado\thanks{Email: sbf@fep.up.pt The work of this author was
supported by FCT- Funda\c{c}\~{a}o para a Ci\^{e}ncia e Tecnologia, under
project UIDB/04561/2020.} \thanks{Corresponding author.}\\CMAFcIO and Faculdade de Economia \\Universidade do Porto\\Rua Dr. Roberto Frias\\4200-464 Porto, Portugal
\and Charles R. Johnson \thanks{Email: crjohn@wm.edu. }\\Department of Mathematics\\College of William and Mary\\Williamsburg, VA 23187-8795}
\maketitle

\begin{abstract}
We focus upon the relationship between Hamiltonian cycle products and
efficient vectors for a reciprocal matrix $A$, to more deeply understand the
latter. This facilitates a new description of the set of efficient vectors (as
a union of convex subsets), greater understanding of convexity within this set
and of order reversals in efficient vectors. A straightforward description of
all efficient vectors for an $n$-by-$n$, column perturbed consistent matrix is
given; it is the union of at most $(n-1)$ choose $2$ convex sets.

\end{abstract}

\textbf{Keywords}: cycle products, decision analysis, efficient vector,
Hamiltonian cycle, pair-wise comparisons, reciprocal matrix

\textbf{MSC2020}: 15B48, 05C20, 90B50, 91B06

\bigskip

\section{Introduction}

An $n$-by-$n$, entry-wise positive matrix $A=[a_{ij}]$ is called
\emph{reciprocal} if $a_{ji}=\frac{1}{a_{ij}}$ for $1\leq i,j\leq n.$ Let
$\mathcal{PC}_{n}$ denote the set of all $n$-by-$n$ reciprocal matrices. Such
matrices represent independent, pair-wise, ratio comparisons among $n$
alternatives. In a variety of models employing reciprocal matrices $A$, a
cardinal ranking vector $w=\left[
\begin{array}
[c]{cccc}%
w_{1} & w_{2} & \cdots & w_{n}%
\end{array}
\right]  ^{T}\in\mathbb{R}_{+}^{n}$ (the entry-wise positive vectors in
$\mathbb{R}^{n}$) is to be deduced from $A,$ so that ratios from $w$
approximate the ratio comparisons in $A$
\cite{anh,choo,dij,fichtner86,golany,is,Kula,zeleny}$.$

Vector $w$ is said to be \emph{efficient} for $A\in\mathcal{PC}_{n}$ if, for
$v\in\mathbb{R}_{+}^{n}$, $\left\vert A-vv^{(-T)}\right\vert \leq\left\vert
A-ww^{(-T)}\right\vert $ implies that $v$ is proportional to $w.$ Here,
$w^{(-T)}$ and $v^{(-T)}$ represent the entry-wise inverse of the transpose of
vectors $w$ and $v$, respectively, $\leq$ an entry-wise inequality and
$\left\vert \cdot\right\vert $ the entry-wise absolute value. Denote the set
of efficient vectors for reciprocal $A$ by $\mathcal{E}(A).$ This is a
connected \cite{blanq2006} but not necessarily convex set. A\ cardinal ranking
vector should be chosen from $\mathcal{E}(A),$ else there is a better
approximating vector. It is now known that the entry-wise geometric convex
hull of the columns of $A$ is contained in $\mathcal{E}(A)$ \cite{FJ2}. In
particular, every column of $A$ is efficient \cite{FJ1}, as well as the simple
entry-wise geometric mean of all columns \cite{blanq2006}. However, for
$n\geq4,$ the right Perron vector \cite{HJ} of $A,\ $the original proposal for
the ranking vector \cite{saaty1977,Saaty}, may or may not be efficient
(depending upon $A$). Classes of reciprocal matrices for which it is
\cite{p6,p2,FerFur}, and for which it is not \cite{bozoki2014,FJ5}, have been
identified. In \cite{FJ5} recent developments about the efficiency of the
Perron vector were provided.

There is a graph theoretic method to decide whether $w\in\mathcal{E}(A).$ The
graph $G(A,w)$ with vertex set $\{1,\ldots,n\}$ has an edge $i\rightarrow j,$
$1\leq i\neq j\leq n,$ if and only if $w_{i}\geq a_{ij}w_{j}$. First in
\cite{blanq2006}, and then in \textrm{\cite{FJ2}} in a simple matricial way,
it was shown that $w\in\mathcal{E}(A)$ if and only if $G(A,w)$ is strongly
connected. Since $A\ $is reciprocal, for every pair of distinct indices
$i,j\in\{1,2,\ldots,n\},$ at least one of the edges $i\rightarrow j$ or
$j\rightarrow i$ is in $G(A,w).$ A directed graph with at least one of
$\{i\rightarrow j,$ $j\rightarrow i\}$ as an edge (and, perhaps both) is
called \emph{semi-complete} \cite{bang} (if it never happens that both occur,
it is called a tournament \cite{bang,Moonbook}); so $G(A,w)$ is always
semi-complete. A semi-complete digraph is strongly connected if and only if it
contains a full cycle \cite{bang}, i.e. a Hamiltonian cycle. We add this fact
here to the graph theoretic characterization of efficiency.

\begin{theorem}
\label{blanq}Suppose that $A\in\mathcal{PC}_{n}$ and that $w\in\mathbb{R}%
_{+}^{n}$. The following are equivalent:

\begin{enumerate}
\item[(i)] $w\in\mathcal{E}(A);$

\item[(ii)] $G(A,w)$ is a strongly connected digraph;

\item[(iii)] $G(A,w)$ contains a Hamiltonian cycle.
\end{enumerate}
\end{theorem}

If there is a $w\in\mathbb{R}_{+}^{n}$ such that $A=ww^{(-T)},$ then $A$ is
called \emph{consistent} (otherwise, it is said to be \emph{inconsistent}). Of
course, a consistent matrix is reciprocal and $\mathcal{E}(A)$ just consists
of positive multiples of $w.$ Any matrix in $\mathcal{PC}_{2}$ is consistent.
Throughout, we focus on the study of $\mathcal{E}(A)$ when $A\in
\mathcal{PC}_{n}$ is inconsistent (and, thus, $n\geq3$)$.$

In \textrm{\cite{FJ3}}, a way to generate all vectors in $\mathcal{E}(A)$
(inductively) was given. Here we follow a different approach to characterize
$\mathcal{E}(A).$ We connect the Hamiltonian cycles in $G(A,w)$ with the
Hamiltonian cycle products $\leq1$ from matrix $A.$ By a (Hamiltonian) cycle
in $A$ we mean a sequence of $n$ entries $a_{ij}$ for which $i\rightarrow j$
are the edges of a Hamiltonian cycle; the product of these entries is the
(Hamiltonian) cycle product. There are at most $\frac{(n-1)!}{2}$ of these
products $<1$, and exactly this number if no cycle product is $1$. We show
that $\mathcal{E}(A)$ is the union of at most $\frac{(n-1)!}{2}$ convex
subsets. Each subset is associated with a Hamiltonian cycle product from $A<1$
(when $A$ is inconsistent). We give sufficient conditions for the convexity of
$\mathcal{E}(A),$ and for the cone generated by the columns of $A$ (that is,
the set of nonzero, linear combinations of the columns of $A$ with nonnegative
coefficients) to lie in $\mathcal{E}(A).$ The former implies the latter since
any column of $A$ lies in $\mathcal{E}(A)$ \textrm{\cite{FJ1}}$.$ If the
latter occurs, the efficiency of the Perron vector of $A$ and of the left
singular vector of $A$ (that is, the Perron vector of $AA^{T}$) follows, as
these vectors belong to the cone generated by the columns \textrm{\cite{FJ4}}.

Using our new description of the efficient vectors for a reciprocal matrix, we
give necessary and sufficient conditions for the existence of an efficient
vector associated with a Hamiltonian cycle to have no order reversals with
entries along the cycle and show that, in any circumstance, there is always
one efficient vector with at most one such order reversal. An efficient vector
$w$ for $A\in\mathcal{PC}_{n}$ is said to exhibit an \emph{order reversal} at
$i,j$ if $w_{i}>w_{j}$ ($w_{i}<w_{j}$) when $a_{ij}<1$ ($a_{ij}>1$)$,$
$w_{i}=w_{j}$ when $a_{ij}\neq1$, or $w_{i}\neq w_{j}$ when $a_{ij}=1.$ There
is now a considerable literature on order reversal and vectors that minimize
the number of order reversals (see, for example, \cite{Csato,FARA,wang}).

If one off-diagonal entry of a consistent matrix, and its symmetrically placed
entry, are changed, so that the resulting $A$ lies in $\mathcal{PC}_{n},$ then
$A$ is called a \emph{simple perturbed consistent matrix}. In prior work
\textrm{\cite{CFF}}, $\mathcal{E}(A)$ for any simple perturbed consistent
matrix $A\in\mathcal{PC}_{n}$ was determined, and this set is defined by a
finite system of linear inequalities on the entries of the vectors, implying
that $\mathcal{E}(A)$ is convex for simple perturbed consistent matrices $A$.
Since this class of matrices includes $\mathcal{PC}_{3},$ it follows that
$\mathcal{E}(A)$ is convex for any $A\in\mathcal{PC}_{3}.$ Recently, in
\textrm{\cite{SZ2}}, the authors illustrated geometrically how $\mathcal{E}%
(A)$ is the union of $3$ convex sets in the $4$-by-$4$ case. These facts
follow from the main result in this paper (Theorem \ref{tt2}).

In \textrm{\cite{FJ4}} we have given examples of matrices obtained from
consistent matrices by changing one column (and the corresponding row
reciprocally), called \emph{column perturbed consistent matrices}, for which
the set of efficient vectors is not convex. However, it was shown that the
cone generated by the columns of such matrices $A$ is contained in
$\mathcal{E}(A).$ Here we describe the set of efficient vectors for a column
perturbed consistent matrix $A\in\mathcal{PC}_{n}$ and show that it is the
union of at most $(n-1)(n-2)/2$ convex sets. The simple perturbed consistent
matrices and some type of double perturbed consistent matrices are special
cases of column perturbed consistent matrices. As mentioned above, the set of
efficient vectors for the former was studied in \cite{CFF} and is convex. The
one for the latter was obtained in \textrm{\cite{Fu22}}. We give here an
example illustrating that it may be not convex. We also present an example of
a matrix $A$ that is neither a simple nor a double perturbed consistent matrix
and for which $\mathcal{E}(A)$ is convex. In the more recent paper
\cite{block}, we have described the efficient vectors for reciprocal matrices
obtained from a consistent matrix by modifying a $3$-by-$3$ principal
submatrix and have provided a class of efficient vectors when the modified
block is of size greater than $3$.

\bigskip

A useful observation is that the set $\mathcal{PC}_{n}$ is closed under both
positive diagonal similarity and permutation similarity (that is, monomial
similarity). Fortunately, such transformations interface with efficient
vectors in a natural way.

\begin{lemma}
\label{lsim}\cite{FJ1} Suppose that $A\in\mathcal{PC}_{n}$ and $w\in
\mathcal{E}(A).$ Then, if $D$ is an $n$-by-$n$ positive diagonal matrix ($P$
is an $n$-by-$n$ permutation matrix), then $Dw\in\mathcal{E}(DAD^{-1})$
($Pw\in\mathcal{E}(PAP^{T})$).
\end{lemma}

\bigskip Note that the graphs $G(A,w)$ and $G(DAD^{-1},Dw)$ coincide. From the
lemma we have that, if $S$ is an $n$-by-$n$ monomial matrix, then
$\mathcal{E}(SAS^{-1})=S\mathcal{E}(A).$

\bigskip

In the next section, we present a fundamental cycle theorem for efficient
vectors of reciprocal matrices. In Section \ref{sorderrev}, we apply that
theorem to give some results on order reversals in an efficient vector. Then,
in Section \ref{s5} we apply the fundamental cycle theorem to study the
efficient vectors for a column perturbed consistent matrix. Some conclusions
are presented in Section \ref{sconclusions}.

\section{The fundamental cycle theorem for efficient vectors of reciprocal
matrices\label{s3}}

Here, we present a fundamental theorem, based upon cycles, that explains
$\mathcal{E}(A)$ for $A\in\mathcal{PC}_{n}.$ This permits insight into the
nature of $\mathcal{E}(A),$ including a sufficient condition for convexity.

\subsection{Auxiliary results}

Next, we note that, for a fixed $A$, the set of all $w$ such that $G(A,w)$
share a common edge is convex.

\begin{lemma}
\label{ll1}Suppose that $A\in\mathcal{PC}_{n}\ $and that $u,v\in\mathbb{R}%
_{+}^{n}$. If $G(A,u)$ and $G(A,v)$ share a common edge, that edge must also
occur in $G(A,w)$ for any $w=tu+(1-t)v,$ $t\in(0,1).$
\end{lemma}

\begin{proof}
Suppose that $G(A,u)$ and $G(A,v)$ have a common edge $i\rightarrow j.$ Then
$u_{i}\geq a_{ij}u_{j}$ and $v_{i}\geq a_{ij}v_{j},$ so that
\begin{align*}
w_{i}  &  =tu_{i}+(1-t)v_{i}\geq ta_{ij}u_{j}+(1-t)a_{ij}v_{j}\\
&  =a_{ij}(tu_{j}+(1-t)v_{j})=a_{ij}w_{j}.
\end{align*}

\end{proof}

As an immediate consequence of Lemma \ref{ll1} and Theorem \ref{blanq}, we
have the following.

\begin{lemma}
\label{ll2}Suppose that $A\in\mathcal{PC}_{n}$ and that $u,v\in\mathbb{R}%
_{+}^{n}$. If $G(A,u)$ and $G(A,v)$ share a common Hamiltonian cycle, then the
line segment joining $u$ and $v$ lies in $\mathcal{E}(A).$
\end{lemma}

The following lemma on Hamiltonian cycles in $A$ along which all entries are
$1$ will be helpful.

\begin{lemma}
\label{lsingleton0}Suppose that $A\in\mathcal{PC}_{n}$ is inconsistent and
$\mathcal{C}$ is a Hamiltonian cycle in $A.$ If all entries in $A$ along
$\mathcal{C}$ are $1,$ then there is a Hamiltonian cycle $\mathcal{C}^{\prime
}$ in $A$ such that all entries in $A$ along $\mathcal{C}^{\prime}$ are
$\leq1$ and there is at least one entry $<1.$
\end{lemma}

\begin{proof}
Since a permutation similarity on $A$ keeps the sets of entries in $A$ along
the Hamiltonian cycles the same, we may assume, without loss of generality,
that $\mathcal{C}$ is the cycle $1\rightarrow2\rightarrow3\rightarrow
\cdots\rightarrow n-1\rightarrow n\rightarrow1.$ Then, $a_{j,j+1}%
=a_{j+1,j}=a_{1n}=a_{n1}=1,$ for $j=1,\ldots,n-1.$ Since $A$ is not
consistent, it has an entry $<1.$ Let $k$ $(>1)$ be the smallest integer such
that either the $k$-th upper-diagonal or the $k$-th lower-diagonal of $A$ has
an entry $<1.$ Then, in one of such diagonals, one of the following occurs: i)
there is an entry $<1$ followed by an entry $\leq1$; ii) after some possible
entries equal to $1,$ the entries alternate between $<1$ and $>1,$ with the
last entry $<1.$ Without loss of generality, we may assume that one of these
situations occurs for the $k$-th upper-diagonal (as otherwise we may consider
$A^{T}$ instead of $A$). Note that $a_{pq}=1$ for $|p-q|<k.$

Case i) Suppose that $a_{i,i+k}<1$ and $a_{i+1,i+1+k}\leq1$ for some
$i\in\{1,\ldots,n-k-1\}.$ Let
\begin{align*}
\mathcal{C}^{\prime}  &  :i\rightarrow i+k\rightarrow i+1\rightarrow
i+2\rightarrow\cdots\rightarrow i+k-1\rightarrow i+k+1\rightarrow\\
i+k+2  &  \rightarrow\cdots\rightarrow n\rightarrow1\rightarrow2\rightarrow
\cdots\rightarrow i-1\rightarrow i\text{.}%
\end{align*}
Note that, if $k=2,$ then $a_{i+k-1,i+k+1}=a_{i+1,i+1+k}\leq1.$

Case ii) We consider two subcases.

Case iia) Suppose that $k>2$ and $a_{n-k,n}<1.$ Let
\[
\mathcal{C}^{\prime}:n-k\rightarrow n\rightarrow1\rightarrow2\rightarrow
\cdots\rightarrow n-k-1\rightarrow n-k+1\rightarrow n-k+2\rightarrow
\cdots\rightarrow n-1\rightarrow n-k\text{. }%
\]

Case iib) Suppose that $k=2$, $a_{n-2,n}<1$, $a_{n-3,n-1}\geq1$,
$a_{n-4,n-2}\leq1$, $a_{n-5,n-3}\geq1,$ etc. Let%
\begin{align*}
\mathcal{C}^{\prime}  &  :1\rightarrow2\rightarrow4\rightarrow\cdots
\rightarrow n-2\rightarrow n\rightarrow n-1\rightarrow n-3\rightarrow
\cdots\rightarrow3\rightarrow1,\text{ if }n\text{ is even,}\\
\mathcal{C}^{\prime}  &  :1\rightarrow3\rightarrow5\rightarrow\cdots
\rightarrow n-2\rightarrow n\rightarrow n-1\rightarrow n-3\rightarrow
\cdots\rightarrow4\rightarrow2\rightarrow1,\text{ if }n\text{ is odd.}%
\end{align*}
In each case, the cycle $\mathcal{C}^{\prime}$ verifies the claim.
\end{proof}

\subsection{Description of the set $\mathcal{\varepsilon}_{A}(\mathcal{C})$}

In what follows, given $A\in\mathcal{PC}_{n}$ and a Hamiltonian cycle
$\mathcal{C}$, we denote by $\mathcal{\varepsilon}_{A}(\mathcal{C})$ the set
of vectors $w$ such that $G(A,w)$ contains the cycle $\mathcal{C}$. Note that,
by Theorem \ref{blanq}, $\varepsilon_{A}(\mathcal{C})\subseteq\mathcal{E}(A)$
and, by Lemma \ref{ll2}, $\varepsilon_{A}(\mathcal{C})$ is convex$.$ Next, we
describe the sets $\mathcal{\varepsilon}_{A}(\mathcal{C}).$ We denote by
$\pi_{1}(A)$ and $\pi_{<1}(A)$ the set of all Hamiltonian cycles for $A$ with
product $\leq1$ and $<1$, respectively. Of course $\pi_{<1}(A)\subseteq\pi
_{1}(A).$

Suppose that $\mathcal{C}$ is the Hamiltonian cycle
\begin{equation}
\gamma_{1}\rightarrow\gamma_{2}\rightarrow\gamma_{3}\rightarrow\cdots
\rightarrow\gamma_{n}\rightarrow\gamma_{1}. \label{cyclegama}%
\end{equation}
Then,
\[
a_{\gamma_{n}\gamma_{1}}%
{\displaystyle\prod\limits_{i=1}^{n-1}}
a_{\gamma_{i}\gamma_{i+1}}%
\]
is the cycle product for $\mathcal{C}$ in $A.$ We have that $w\in
\mathcal{\varepsilon}_{A}(\mathcal{C})$ if and only if
\begin{equation}%
\begin{array}
[c]{lll}%
w_{\gamma_{1}}\geq a_{\gamma_{1}\gamma_{2}}w_{\gamma_{2}},\quad & \quad
w_{\gamma_{2}}\geq a_{\gamma_{2}\gamma_{3}}w_{\gamma_{3}}, & \ldots,\\
&  & \\
\ldots,\quad & w_{\gamma_{n-1}}\geq a_{\gamma_{n-1}\gamma_{n}}w_{\gamma_{n}%
}, & \quad w_{\gamma_{n}}\geq a_{\gamma_{n}\gamma_{1}}w_{\gamma_{1}},
\end{array}
\label{f111}%
\end{equation}
or, equivalently,%
\begin{align}
w_{\gamma_{1}}  &  \geq a_{\gamma_{1}\gamma_{2}}w_{\gamma_{2}}\geq\left(
{\displaystyle\prod\limits_{i=1}^{2}}
a_{\gamma_{i}\gamma_{i+1}}\right)  w_{\gamma_{3}}\geq\nonumber\\
& \label{f222}\\
\cdots &  \geq\left(
{\displaystyle\prod\limits_{i=1}^{n-1}}
a_{\gamma_{i}\gamma_{i+1}}\right)  w_{\gamma_{n}}\geq\left(  a_{\gamma
_{n}\gamma_{1}}%
{\displaystyle\prod\limits_{i=1}^{n-1}}
a_{\gamma_{i}\gamma_{i+1}}\right)  w_{\gamma_{1}}.\nonumber
\end{align}
For a matricial description of $\varepsilon_{A}(\mathcal{C}),$ consider
\[
P=\left[
\begin{tabular}
[c]{c|c}%
$%
\begin{array}
[c]{c}%
0\\
\vdots\\
0
\end{array}
$ & $I_{n-1}$\\\hline
$1$ & $%
\begin{array}
[c]{ccc}%
0 & \cdots & 0
\end{array}
$%
\end{tabular}
\ \ \ \ \ \ \ \ \ \ \right]  \text{ and }S=\operatorname*{diag}(a_{\gamma
_{1}\gamma_{2}},\text{ }a_{\gamma_{2}\gamma_{3}},\text{ }\ldots,\text{
}a_{\gamma_{n-1}\gamma_{n}},\text{ }a_{\gamma_{n}\gamma_{1}}).
\]
Then (\ref{f111}) is equivalent to $SPw_{\gamma}\leq w_{\gamma},$ in which
$w_{\gamma}=\left[
\begin{array}
[c]{cccc}%
w_{\gamma_{1}} & w_{\gamma_{2}} & \cdots & w_{\gamma_{n}}%
\end{array}
\right]  ^{T}.$ Thus $w\in\mathcal{\varepsilon}_{A}(\mathcal{C})$ if and only
if $w_{\gamma}$ satisfies the system of linear inequalities
\[
(I_{n}-SP)w_{\gamma}\geq0\text{, }w_{\gamma}>0.
\]

We say that a set of efficient vectors for $A$ is a \emph{singleton} if it
only contains a single (positive) vector (and all its positive multiples).

\begin{lemma}
\label{lempty}Suppose that $A\in\mathcal{PC}_{n}$ and that $\mathcal{C},$ as
in (\ref{cyclegama}), lies in $\pi_{1}(A)$. The set of positive solutions $w$
to (\ref{f111}), when any $n-1$ inequalities are taken to be equalities, is a
singleton. Moreover, a vector in each of the $n$ singletons is extreme in
$\mathcal{\varepsilon}_{A}(\mathcal{C}),$ so that $\mathcal{\varepsilon}%
_{A}(\mathcal{C})$ is the cone generated by these $n$ vectors.
\end{lemma}

\begin{proof}
Let $k\in\{1,\ldots,n\}.$ Suppose that all inequalities in (\ref{f111}) are
taken to be equalities except the $k$-th one. First suppose that $k\neq n.$
Multiplying the equalities, we get
\begin{align*}%
{\displaystyle\prod\limits_{i=1,i\neq k}^{n}}
w_{\gamma_{i}}  &  =a_{\gamma_{n}\gamma_{1}}w_{\gamma_{1}}%
{\displaystyle\prod\limits_{i=1,i\neq k}^{n-1}}
\left(  a_{\gamma_{i}\gamma_{i+1}}w_{\gamma_{i+1}}\right) \\
&  \Leftrightarrow w_{\gamma_{k+1}}=\left(  a_{\gamma_{n}\gamma_{1}}%
{\displaystyle\prod\limits_{i=1}^{n-1}}
a_{\gamma_{i}\gamma_{i+1}}\right)  a_{\gamma_{k+1}\gamma_{k}}w_{\gamma_{k}}.
\end{align*}
Since
\[
a_{\gamma_{n}\gamma_{1}}%
{\displaystyle\prod\limits_{i=1}^{n-1}}
a_{\gamma_{i}\gamma_{i+1}}\leq1,
\]
we get
\[
w_{\gamma_{k+1}}\leq a_{\gamma_{k+1}\gamma_{k}}w_{\gamma_{k}}\Leftrightarrow
w_{\gamma_{k}}\geq a_{\gamma_{k}\gamma_{k+1}}w_{\gamma_{k+1}},
\]
implying that the $k$-th inequality is satisfied. If $k=n,$ it can be seen in
a similar way that the equalities taken from the first $n-1$ inequalities in
(\ref{f111}) imply the $n$-th inequality. Clearly, taking one entry $1,$ each
$n-1$ equalities in (\ref{f111}) determines a unique positive solution for
$w.$ So, we get $n$ (some possibly equal) vectors in $\mathcal{\varepsilon
}_{A}(\mathcal{C})$ that are precisely the extreme points of the set.
\end{proof}

\bigskip

Next we show an important fact about the structure of $\mathcal{\varepsilon
}_{A}(\mathcal{C})$.

\begin{lemma}
\label{lltt}Suppose that $A\in\mathcal{PC}_{n}$ and that $\mathcal{C}\in
\pi_{1}(A)$. Then $\mathcal{\varepsilon}_{A}(\mathcal{C})$ is a singleton if
and only if the cycle product for $\mathcal{C}$ in $A$ is $1.$ Moreover, if
$A$ is inconsistent and the product for $\mathcal{C}$ in $A$ is $1,$ then
$\varepsilon_{A}(\mathcal{C})\subseteq\varepsilon_{A}(\mathcal{C}^{\prime}),$
for some $\mathcal{C}^{\prime}\in\pi_{<1}(A).$
\end{lemma}

\begin{proof}
Since a diagonal similarity does not change any cycle product in $A$, and by a
diagonal similarity on $A$ any $n-1$ entries along $\mathcal{C}$ may be made
$1$ (and the remaining entry is $\leq1,$ since $\mathcal{C}\in\pi_{1}(A)$),
taking into account Lemma \ref{lsim}, we assume that this situation occurs in
order to prove the result.

When the product for $\mathcal{C}$ in $A$ is $1,$ all entries in $A$ along
$\mathcal{C}$ are $1.$ Then, from (\ref{f111}) (with $\mathcal{C}$ as in
(\ref{cyclegama})), $w\in\varepsilon_{A}(\mathcal{C})$ if and only if, up to a
positive factor, $w$ is the vector of $1$'s, implying that
$\mathcal{\varepsilon}_{A}(\mathcal{C})$ is a singleton. When the product of
$\mathcal{C}$ in $A$ is $<$ $1,$ by Lemma \ref{lempty}, there is a $w$
satisfying (\ref{f111}) in which the first $n-1$ inequalities are taken to be
equalities. Then, the $n$-th inequality is strict. Any vector $w^{\prime}$
obtained from $w$ by a sufficiently small decrease of $w_{\gamma_{n}},$ and
agreeing with $w$ in all other components, still satisfies (\ref{f111}) (with
$w_{\gamma_{n}}^{\prime}$ instead of $w_{\gamma_{n}}$). So
$\mathcal{\varepsilon}_{A}(\mathcal{C})$ is not a singleton. This completes
the proof of the first claim.

Suppose that the product along $\mathcal{C}$ in $A$ (inconsistent) is $1,$ in
which case all entries in $A\ $along $\mathcal{C}$ are $1$ (as we are assuming
that there are $n-1$ such entries equal to $1$). Thus, all vectors $w$ in
$\varepsilon_{A}(\mathcal{C})$ are constant. By Lemma \ref{lsingleton0}, there
is a Hamiltonian cycle $\mathcal{C}^{\prime}$ in $A$ such that all entries in
$A$ along $\mathcal{C}^{\prime}$ are $\leq1$ and there is at least one entry
$<1.$ For $\mathcal{C}^{\prime}$ as in (\ref{cyclegama}), it follows that $w$
satisfies (\ref{f111}), implying that $w\in\varepsilon_{A}(\mathcal{C}%
^{\prime}).$ Since $\mathcal{C}^{\prime}\in\pi_{<1}(A),$ the last claim follows.
\end{proof}

\bigskip

We observe that, if the product along a Hamiltonian cycle $\mathcal{C}$ in $A$
is $1,$ then so is the product along the reverse cycle $\mathcal{C}^{r}$ of
$\mathcal{C},$ and $\mathcal{\varepsilon}_{A}(\mathcal{C}%
)=\mathcal{\varepsilon}_{A}(\mathcal{C}^{r})$.

\subsection{Main result}

Now, we may give the fundamental theorem on the cycle structure of reciprocal
matrices and their efficient vectors. If $A\in\mathcal{PC}_{n}$ is consistent,
then $\mathcal{E}(A)$ is the singleton generated by any column of $A,$
otherwise, since each column is in $\mathcal{E}(A)$ and not all columns are
proportional, $\mathcal{E}(A)$ is not a singleton. We assume that $A$ is inconsistent.

\begin{theorem}
\label{tt2}Suppose that $A\in\mathcal{PC}_{n}$ is inconsistent. If
$w\in\mathcal{E}(A)$ and $\mathcal{C}$ is a Hamiltonian cycle in $G(A,w),$
then $\mathcal{C}\in\pi_{1}(A).$ On the other hand, if $\mathcal{C}\in\pi
_{1}(A),$ then the set
\[
\mathcal{\varepsilon}_{A}(\mathcal{C})=\left\{  w:G(A,w)\text{ contains
}\mathcal{C}\right\}
\]
is a nonempty, convex subset of $\mathcal{E}(A).$ Moreover,
\begin{equation}
\mathcal{E}(A)=%
{\displaystyle\bigcup\limits_{\mathcal{C}\in\pi_{<1}(A)}}
\mathcal{\varepsilon}_{A}(\mathcal{C}). \label{union}%
\end{equation}

\end{theorem}

\begin{proof}
If $w\in\mathcal{E}(A)$ and $\mathcal{C}$, as in (\ref{cyclegama}), is a
Hamiltonian cycle in $G(A,w),$ then $w\in\varepsilon_{A}(\mathcal{C}%
).\ $Hence, (\ref{f222}) holds and the inequality between the left most and
the right most expressions imply
\[
a_{\gamma_{n}\gamma_{1}}%
{\displaystyle\prod\limits_{i=1}^{n-1}}
a_{\gamma_{i}\gamma_{i+1}}\leq1,
\]
which means that $\mathcal{C}\in\pi_{1}(A).$

Suppose that $\mathcal{C}\in\pi_{1}(A).$ The set $\varepsilon_{A}%
(\mathcal{C})$ is nonempty by Lemma \ref{lempty}. As, for $w\in
\mathcal{\varepsilon}_{A}(\mathcal{C})$, $G(A,w)$ contains the cycle
$\mathcal{C}$, by Theorem \ref{blanq}, $w\in\mathcal{E}(A).$ Also, note that,
as all such $G(A,w)$ contain $\mathcal{C}$, the set is convex by Lemma
\ref{ll2} (alternatively, (polyhedral) convexity follows from the fact that
$w\in\mathcal{\varepsilon}_{A}(\mathcal{C})$ is determined by the finite
system of linear inequalities (\ref{f111})). Since, by Theorem \ref{blanq},
each $w\in\mathcal{E}(A)\ $lies in $\varepsilon_{A}(\mathcal{C})$ for some
$\mathcal{C}$, and, by the first part of the proof, $\mathcal{C}\in\pi_{1}%
(A)$, it follows that
\begin{equation}
\mathcal{E}(A)=%
{\displaystyle\bigcup\limits_{\mathcal{C}\in\pi_{1}(A)}}
\mathcal{\varepsilon}_{A}(\mathcal{C}). \label{uun}%
\end{equation}
Now (\ref{union}) follows taking into account the last claim in Lemma
\ref{lltt}. This conclusion also follows from the following alternate
argument. Since $\mathcal{E}(A)$ is connected \cite{blanq2006} and is not a
singleton (as $A$ is inconsistent), it follows that any singleton
$\mathcal{\varepsilon}_{A}(\mathcal{C})$ in the union in (\ref{uun}) should be
contained in a non-singleton $\mathcal{\varepsilon}_{A}(\mathcal{C}^{\prime
}),$ which, by Lemma \ref{lltt}, implies that $\mathcal{C}^{\prime}\in\pi
_{<1}(A).$
\end{proof}

\bigskip

Note that $\pi_{<1}(A)$ has no more than $\frac{(n-1)!}{2}$ elements and has
less if there are Hamiltonian cycles from $A$ with product $1$. If $n=3,$ then
$\frac{(n-1)!}{2}=1$ which means that $\mathcal{E}(A)=\varepsilon
_{A}(\mathcal{C})$ for some cycle $\mathcal{C}\in$ $\pi_{1}(A)$ (in fact,
$\mathcal{C}\in\pi_{<1}(A)$ if $A$ is inconsistent)$,$ which is another way to
see that $\mathcal{E}(A)$ is convex for $A\in\mathcal{PC}_{3}.$

\bigskip

We next give some consequences of Theorem \ref{tt2}. Since the results can be
trivially verified for consistent matrices, we state them for general
reciprocal matrices.

\begin{corollary}
\label{tt1}Suppose $A\in\mathcal{PC}_{n}.$ If there exists a Hamiltonian cycle
that lies in $G(A,w)$ for every $w\in\mathcal{E}(A),$ then $\mathcal{E}(A)$ is
convex. In particular, this holds if there is just one Hamiltonian cycle
product $<1$ in $A$.
\end{corollary}

Observe that the first claim in Corollary \ref{tt1} is also a consequence of
Lemma \ref{ll2}.

\begin{corollary}
\label{c1}Suppose $A\in\mathcal{PC}_{n}.$ If all columns of $A$ are in
$\mathcal{\varepsilon}_{A}(\mathcal{C})$ for some $\mathcal{C}\in\pi_{1}(A)$,
then the cone generated by the columns of $A$ is contained in $\mathcal{E}%
(A)$. In particular, the Perron vector and the singular vector of $A$ are
efficient for $A.$
\end{corollary}

It is known that in general $\mathcal{E}(A)$ is not closed under entry-wise
geometric mean \cite{FJ1}. However, we note that each subset
$\mathcal{\varepsilon}_{A}(\mathcal{C})$ is (in fact, it is closed under
entry-wise weighted geometric means), because it is defined by inequalities
(\ref{f111}).

\begin{corollary}
Suppose $A\in\mathcal{PC}_{n}.$ If there exists a Hamiltonian cycle that lies
in $G(A,w)$ for every $w\in\mathcal{E}(A),$ then $\mathcal{E}(A)$ is closed
under entry-wise weighted geometric means.
\end{corollary}

\bigskip We give next an example illustrating the previous results.

\begin{example}
\label{ex1}Let%
\[
A=\left[
\begin{array}
[c]{cccc}%
1 & 2 & 1 & \frac{1}{2}\\
\frac{1}{2} & 1 & 2 & 1\\
1 & \frac{1}{2} & 1 & 2\\
2 & 1 & \frac{1}{2} & 1
\end{array}
\right]  .
\]
Note that there are $2$ cycle products in $A$ equal to $1,$%
\begin{align*}
\mathcal{C}_{1}  &  :1\rightarrow2\rightarrow4\rightarrow3\rightarrow1,\\
\mathcal{C}_{2}  &  :1\rightarrow3\rightarrow2\rightarrow4\rightarrow1,
\end{align*}
and their reverses, and the cycle product $<1$%
\[
\mathcal{C}_{3}:1\rightarrow4\rightarrow3\rightarrow2\rightarrow1.
\]
By Theorem \ref{tt2}, $\mathcal{E}(A)=\mathcal{\varepsilon}_{A}(\mathcal{C}%
_{3}),$ and, so, $\mathcal{E}(A)$ is convex. In particular, the Perron vector
and the singular vector of $A$ are efficient for $A.$ Also, the set
$\mathcal{E}(A)$ is closed under entry-wise weighted geometric means. Note
that all the columns of $A$ lie in $\mathcal{\varepsilon}_{A}(\mathcal{C}%
_{3})$, as they are efficient for $A.$ Using (\ref{f111}), we can see that
\[
\mathcal{E}(A)=\mathcal{\varepsilon}_{A}(\mathcal{C}_{3})=\left\{
w\in\mathbb{R}_{+}^{4}:w_{1}\geq\frac{1}{2}w_{4}\geq\frac{1}{4}w_{3}\geq
\frac{1}{8}w_{2}\geq\frac{1}{16}w_{1}\right\}  .
\]
We also have%
\begin{align*}
\mathcal{\varepsilon}_{A}(\mathcal{C}_{1})  &  =\left\{  w\in\mathbb{R}%
_{+}^{4}:w_{1}=2w_{2}=2w_{4}=w_{3}\right\}  ,\\
\mathcal{\varepsilon}_{A}(\mathcal{C}_{2})  &  =\left\{  w\in\mathbb{R}%
_{+}^{4}:w_{1}=w_{3}=\frac{1}{2}w_{2}=\frac{1}{2}w_{4}\right\}  ,
\end{align*}
which can be verified to be contained in $\mathcal{\varepsilon}_{A}%
(\mathcal{C}_{3}),$ as expected. We observe that $A$ is not a simple perturbed
consistent matrix, as it is nonsingular. Thus, this example illustrates that
the convexity of $\mathcal{E}(A)$ may occur for other matrices $A$ than the
simple perturbed consistent matrices. The matrix $A$ is also not double perturbed.
\end{example}

\section{Order reversals\label{sorderrev}}

Here we give some consequences of Theorem \ref{tt2} regarding the existence of
order reversals in an efficient vector. Of course, if $A$ is consistent, no
efficient vector for $A$ exhibits an order reversal, as $\frac{w_{i}}{w_{j}%
}=a_{ij}$ for all $i,j.$

\begin{theorem}
\label{trev}Suppose that $A\in\mathcal{PC}_{n}$ is inconsistent and
$\mathcal{C}\in\pi_{1}(A).$ Then, there is a $w\in\varepsilon_{A}%
(\mathcal{C)}$ that exhibits no order reversal with entries of $A$ along
$\mathcal{C}$ if and only if either there is an entry of $A$ $>1$ along
$\mathcal{C}$ or the cycle product is $1.$ Otherwise, there is a
$w\in\varepsilon_{A}(\mathcal{C)}$ with exactly $1$ order reversal along
$\mathcal{C}$.
\end{theorem}

\begin{proof}
Suppose that there is a $w\in\varepsilon_{A}(\mathcal{C)}$ that exhibits no
order reversal with entries of $A$ along $\mathcal{C}$ and all entries along
$\mathcal{C}$ are $\leq1,$ with at least one inequality being strict. Then
$\frac{w_{i}}{w_{j}}$ would be $\leq1$ for the efficient vector $w$ and all
$i,j$ along the cycle, with at least one inequality being strict, implying%
\[
1=%
{\displaystyle\prod\limits_{i,j\text{ along }\mathcal{C}}}
\frac{w_{i}}{w_{j}}<1,
\]
a contradiction. So, an entry $>1$ or all entries $1$ along the cycle is necessary.

Suppose there is an entry $>1$ or all entries are $1$ along $\mathcal{C}$. For
$\mathcal{C}$ as in (\ref{cyclegama}), suppose, without loss of generality,
that $a_{\gamma_{n}\gamma_{1}}\geq1$, with $a_{\gamma_{n}\gamma_{1}}=1$ if and
only if all entries along the cycle are $1.$ Now, use the inequalities
(\ref{f111}) and take the first $n-1$ of them to be equalities, so as to
define $w,$ taking $w_{1}=1$ (Lemma \ref{lempty}). The last inequality (which
is implied by the first $n-1$ inequalities) then ensures that there is no
order reversal at $\gamma_{n},\gamma_{1}$. In fact, $a_{\gamma_{n}\gamma_{1}%
}>1$ implies $\frac{w_{\gamma_{n}}}{w_{\gamma_{1}}}>1$. Also, $a_{\gamma
_{n}\gamma_{1}}=1$ implies, by hypothesis, all entries of $A$ along the cycle
equal to $1,$ and thus, by the $n-1$ equalities, $\frac{w_{\gamma_{j+1}}%
}{w_{\gamma_{j}}}=1,$ $1\leq j\leq n-1.$ Then, $\frac{w_{\gamma_{n}}%
}{w_{\gamma_{1}}}=1.$ There is no order reversal elsewhere along the cycle
because of the equalities. The same construction ensures at most one order
reversal without the assumption of an entry $>1$ or all entries $1$ along
$\mathcal{C}.$
\end{proof}

\bigskip

Theorem \ref{trev} concerns the existence of an efficient vector for $A$ with
at most $1$ order reversal along a cycle $\mathcal{C}\in\pi_{1}(A).$ The
vector may exhibit order reversals at positions in the matrix not along the cycle.

\section{Efficient vectors for column perturbed consistent matrices\label{s5}}

In this section, we study the efficient vectors for column perturbed
consistent matrices \textrm{\cite{FJ5,FJ4}.} Based on Lemma \ref{lsim} and the
following observation, we may assume that these matrices have a simple form.
By $J_{k}$ we denote the $k$-by-$k$ matrix with all entries $1.$

\begin{lemma}
If $B\in\mathcal{PC}_{n}$ is a column perturbed consistent matrix, then $B$ is
monomially similar to a matrix of the form%
\begin{equation}
A=\left[
\begin{array}
[c]{cc}%
1 &
\begin{array}
[c]{cccc}%
a_{12} & a_{13} & \cdots & a_{1n}%
\end{array}
\\%
\begin{array}
[c]{c}%
\frac{1}{a_{12}}\\
\frac{1}{a_{13}}\\
\vdots\\
\frac{1}{a_{1n}}%
\end{array}
& J_{n-1}%
\end{array}
\right]  \in\mathcal{PC}_{n}. \label{CPC}%
\end{equation}

\end{lemma}

\begin{proof}
Suppose that $B\ $is obtained from the consistent matrix $ww^{(-T)}$ by
modifying, say, row and column $i.$ Then, for $D^{-1}=\operatorname*{diag}%
(w),$ we have that $DBD^{-1}$ has all entries equal to $1,$ except those in
row and column $i.$ Then $DBD^{-1}$ is permutationally similar to a reciprocal
matrix with all entries equal to $1,$ except those in the first row and column.
\end{proof}

\bigskip

Before we give the description of the efficient vectors for a column perturbed
consistent matrix, we illustrate it with an example.

\begin{example}
Let%
\[
A=\left[
\begin{array}
[c]{ccccc}%
1 & \frac{1}{5} & \frac{1}{4} & 2 & 3\\
5 & 1 & 1 & 1 & 1\\
4 & 1 & 1 & 1 & 1\\
\frac{1}{2} & 1 & 1 & 1 & 1\\
\frac{1}{3} & 1 & 1 & 1 & 1
\end{array}
\right]  .
\]
We have no cycle products equal to $1,$ since $a_{1i}a_{j1}\neq1$ for $i\neq
j.$ So, we have exactly $12$ cycle products $<1.$ They can be written as
\[
\mathcal{C}_{\gamma}:1\rightarrow\gamma_{2}\rightarrow\gamma_{3}%
\rightarrow\gamma_{4}\rightarrow\gamma_{5}\rightarrow1,
\]
with $a_{1\gamma_{2}}a_{\gamma_{5}1}<1.$ Then%
\[
\mathcal{\varepsilon}_{A}(\mathcal{C}_{\gamma})=\left\{  w:w_{1}\geq
a_{1\gamma_{2}}w_{\gamma_{2}}\wedge w_{\gamma_{2}}\geq w_{\gamma_{3}}\geq
w_{\gamma_{4}}\geq w_{\gamma_{5}}\wedge w_{\gamma_{5}}\geq a_{\gamma_{5}%
1}w_{1}\right\}  .
\]
We have that $\gamma$ can be one of the following permutations of
$\{2,3,4,5\}$ (we present them in a convenient order):%
\begin{align*}
\gamma^{(1)}  &  :2453,\quad\gamma^{(2)}:2543,\quad\gamma^{(3)}:2354,\quad
\gamma^{(4)}:2534,\quad\gamma^{(5)}:2345,\quad\gamma^{(6)}:2435,\quad\\
\gamma^{(7)}  &  :3254,\quad\gamma^{(8)}:3524,\quad\gamma^{(9)}:3245,\quad
\gamma^{(10)}:3425,\quad\gamma^{(11)}:4235,\quad\gamma^{(12)}:4325.
\end{align*}
Then,%
\begin{align*}
\mathcal{\varepsilon}_{A}(\mathcal{C}_{\gamma^{(1)}})\cup\mathcal{\varepsilon
}_{A}(\mathcal{C}_{\gamma^{(2)}})  &  =\left\{  w\in\mathbb{R}_{+}^{5}%
:a_{21}w_{1}\geq w_{2}\geq w_{4},w_{5}\geq w_{3}\geq a_{31}w_{1}\right\}  ,\\
\mathcal{\varepsilon}_{A}(\mathcal{C}_{\gamma^{(3)}})\cup\mathcal{\varepsilon
}_{A}(\mathcal{C}_{\gamma^{(4)}})  &  =\left\{  w\in\mathbb{R}_{+}^{5}%
:a_{21}w_{1}\geq w_{2}\geq w_{3},w_{5}\geq w_{4}\geq a_{41}w_{1}\right\}  ,\\
\mathcal{\varepsilon}_{A}(\mathcal{C}_{\gamma^{(5)}})\cup\mathcal{\varepsilon
}_{A}(\mathcal{C}_{\gamma^{(6)}})  &  =\left\{  w\in\mathbb{R}_{+}^{5}%
:a_{21}w_{1}\geq w_{2}\geq w_{3},w_{4}\geq w_{5}\geq a_{51}w_{1}\right\}  ,\\
\mathcal{\varepsilon}_{A}(\mathcal{C}_{\gamma^{(7)}})\cup\mathcal{\varepsilon
}_{A}(\mathcal{C}_{\gamma^{(8)}})  &  =\left\{  w\in\mathbb{R}_{+}^{5}%
:a_{31}w_{1}\geq w_{3}\geq w_{2},w_{5}\geq w_{4}\geq a_{41}w_{1}\right\}  ,\\
\mathcal{\varepsilon}_{A}(\mathcal{C}_{\gamma^{(9)}})\cup\mathcal{\varepsilon
}_{A}(\mathcal{C}_{\gamma^{(10)}})  &  =\left\{  w\in\mathbb{R}_{+}^{5}%
:a_{31}w_{1}\geq w_{3}\geq w_{2},w_{4}\geq w_{5}\geq a_{51}w_{1}\right\}  ,\\
\mathcal{\varepsilon}_{A}(\mathcal{C}_{\gamma^{(11)}})\cup\mathcal{\varepsilon
}_{A}(\mathcal{C}_{\gamma^{(12)}})  &  =\left\{  w\in\mathbb{R}_{+}^{5}%
:a_{41}w_{1}\geq w_{4}\geq w_{2},w_{3}\geq w_{5}\geq a_{51}w_{1}\right\}  .
\end{align*}
By Theorem \ref{tt2},%
\[
\mathcal{E}(A)=%
{\displaystyle\bigcup\limits_{i=1}^{6}}
\left(  \mathcal{\varepsilon}_{A}(\mathcal{C}_{\gamma^{(2i-1)}})\cup
\mathcal{\varepsilon}_{A}(\mathcal{C}_{\gamma^{(2i)}})\right)  .
\]
Note that each set $\mathcal{\varepsilon}_{A}(\mathcal{C}_{\gamma^{(2i-1)}%
})\cup\mathcal{\varepsilon}_{A}(\mathcal{C}_{\gamma^{(2i)}})$ is convex. So,
in this case $\mathcal{E}(A)$ is the union of $6$ convex sets.
\end{example}

For $A\in\mathcal{PC}_{n}$ and $i,j\in\{1,\ldots,n\},$ $i\neq j,$ let
\[
\mathcal{\varepsilon}_{ij}(A)=\left\{  w\in\mathbb{R}_{+}^{n}:a_{i1}w_{1}\geq
w_{i}\geq w_{k}\geq w_{j}\geq a_{j1}w_{1}\text{, }k\neq1,i,j\right\}  .
\]

\bigskip

Since $w$ in $\mathcal{\varepsilon}_{ij}(A)$ is defined by a finite number of
linear inequalities in its (positive) entries, we have the following.

\begin{lemma}
\label{lconv}For $A\in\mathcal{PC}_{n}$ and $i,j\in\{1,\ldots,n\},$ $i\neq j,$
$\mathcal{\varepsilon}_{ij}(A)$ is convex.
\end{lemma}

\begin{lemma}
\label{leij}Let $i,j\in\{2,\ldots,n\},$ $i\neq j.$ If $A\in\mathcal{PC}_{n}$
is as in (\ref{CPC}) with $a_{1i}a_{j1}<1,$ then $\mathcal{\varepsilon}%
_{ij}(A)\subseteq\mathcal{E}(A).$ In particular, $\mathcal{\varepsilon}%
_{ij}(A)$ is the subset of $\mathcal{E}(A)$ that comes from the cycles in
$\pi_{<1}(A)$ of the form $1\rightarrow i\rightarrow\cdots\rightarrow
j\rightarrow1.$
\end{lemma}

\begin{proof}
Denote by $\mathcal{C}_{ij}$ the set of all Hamiltonian cycles $\mathcal{C}$
of the form $1\rightarrow i\rightarrow\gamma_{3}\rightarrow\cdots
\rightarrow\gamma_{n-1}\rightarrow j\rightarrow1$ (there are $(n-3)!$)$.$
Suppose that $a_{1i}a_{j1}<1,$ so that the product in $A$ for any cycle
$\mathcal{C}$ in $\mathcal{C}_{ij}$ is $<1.$ We have%
\[
\mathcal{\varepsilon}_{A}(\mathcal{C})=\left\{  w\in\mathbb{R}_{+}^{n}%
:a_{i1}w_{1}\geq w_{i}\geq w_{\gamma_{3}}\geq\cdots\geq w_{\gamma_{n-1}}\geq
w_{j}\geq a_{j1}w_{1}\right\}  .
\]
Then
\[
\mathcal{\varepsilon}_{ij}(A)=%
{\displaystyle\bigcup\limits_{\mathcal{C\in C}_{ij}}^{n}}
\mathcal{\varepsilon}_{A}(\mathcal{C}).
\]
Since, by Theorem \ref{tt2}, $\mathcal{\varepsilon}_{A}(\mathcal{C}%
)\subseteq\mathcal{E}(A),$ for each $\mathcal{C\in C}_{ij},$ the claim follows.
\end{proof}

\bigskip

We give next the main result of this section.

\begin{theorem}
\label{tuni}If $A\in\mathcal{PC}_{n}$ is inconsistent of the form
(\ref{CPC})$,$ then
\begin{equation}
\mathcal{E}(A)=%
{\displaystyle\bigcup\limits_{(i,j)\in\mathcal{N}}}
\mathcal{\varepsilon}_{ij}(A), \label{unni}%
\end{equation}
for $\mathcal{N}=\{(i,j):i,j\in\{2,\ldots,n\}$ and $a_{1i}a_{j1}<1\}.$
\end{theorem}

\begin{proof}
Let $\mathcal{C}$ be a Hamiltonian cycle in $A.$ By Lemma \ref{leij},
$\mathcal{\varepsilon}_{ij}(A),$ with $(i,j)\in\mathcal{N}$, is contained in
$\mathcal{E}(A),$ proving the inclusion $\supseteq$ in (\ref{unni})$.$ On the
other hand, by Theorem \ref{tt2}, if $w\in\mathcal{E}(A),$ then $w\in
\varepsilon_{A}(\mathcal{C})$ for some $\mathcal{C}\in\pi_{<1}(A).$ Then, for
$\mathcal{C}:1\rightarrow i\rightarrow\gamma_{3}\rightarrow\cdots
\rightarrow\gamma_{n-1}\rightarrow j\rightarrow1$, the cycle product for
$\mathcal{C}$ in $A$ is $a_{1i}a_{j1}<1.$ By Lemma \ref{leij}, $\varepsilon
_{A}(\mathcal{C})\subseteq\mathcal{\varepsilon}_{ij}(A)$.
\end{proof}

\bigskip

The set $\mathcal{N}$ has at most $\frac{(n-1)(n-2)}{2}$ elements and has
exactly this number if $a_{12},\ldots,a_{1n}$ are pairwise distinct$.$ In
fact, the number of elements in $\mathcal{N}$ is the number of pairs $(i,j)$
with $i,j\in\{2,\ldots,n\},$ $j>i,$ and $a_{1i}\neq a_{1j}.$

\begin{corollary}
If $B\in\mathcal{PC}_{n}$ is a general inconsistent column perturbed
consistent matrix, then $\mathcal{E}(B)$ is the union of (at most)
$\frac{(n-1)(n-2)}{2}$ convex sets$.$
\end{corollary}

\begin{proof}
By previous observations, there is an $n$-by-$n$ monomial matrix $S$ such that
$A=S^{-1}BS$ is as in (\ref{CPC}). By Lemma \ref{lsim}, $\mathcal{E}%
(B)=S\mathcal{E}(A).$ Thus, by Theorem \ref{tuni},
\[
\mathcal{E}(B)=%
{\displaystyle\bigcup\limits_{(i,j)\in\mathcal{N}}}
S\mathcal{\varepsilon}_{ij}(A),
\]
with $\mathcal{N}$ as in the theorem. Since, by Lemma \ref{lconv},
$\mathcal{\varepsilon}_{ij}(A)$ is convex then $S\mathcal{\varepsilon}%
_{ij}(A)$ is convex \cite{FJ4}.
\end{proof}

\bigskip

We finally show that the description of the efficient vectors for a simple
perturbed consistent matrix given in \cite{CFF} is an easy consequence of the
results developed in this paper. Suppose that $A\in\mathcal{PC}_{n},$ $n>2,$
is as in (\ref{CPC}), with $a_{12}=\cdots=a_{1,n-1}=1$ and $a_{1n}>1$ (if
$a_{1n}=1,$ $A$ is consistent), which can be assumed for the purpose of
studying the efficient vectors for a simple perturbed consistent matrix, by
Lemma \ref{lsim}. Then, the Hamiltonian cycles whose products from $A$ are
$<1$ are of the form $\mathcal{C}:1\rightarrow\ell\rightarrow\gamma
_{3}\rightarrow\cdots\rightarrow\gamma_{n-1}\rightarrow n\rightarrow1$,
$\ell=2,\ldots,n-1.$ We have%
\[
\mathcal{\varepsilon}_{\ell n}(A)=\left\{  w\in\mathbb{R}_{+}^{n}:w_{1}\geq
w_{\ell}\geq w_{k}\geq w_{n}\geq a_{n1}w_{1}\text{, }k\neq1,\ell,n\right\}  .
\]
Then,%
\begin{equation}%
{\displaystyle\bigcup\limits_{\ell=2}^{n-1}}
\mathcal{\varepsilon}_{\ell n}(A)=\left\{  w\in\mathbb{R}_{+}^{n}:w_{1}\geq
w_{k}\geq w_{n}\geq a_{n1}w_{1}\text{, }k\neq1,n\right\}  . \label{set}%
\end{equation}
By Theorem \ref{tuni}, $\mathcal{E}(A)$ is the set (\ref{set})$,$ as claimed
in \cite{CFF}.

\bigskip

Though the set of efficient vectors for a simple perturbed consistent matrix
is convex, when the reciprocal matrix is double perturbed, that is, is
obtained from a consistent matrix by changing two pairs of reciprocal entries,
non-convexity may occur.

\begin{example}
Let%
\[
A=\left[
\begin{array}
[c]{cccc}%
1 & \frac{1}{3} & \frac{1}{2} & 1\\
3 & 1 & 1 & 1\\
2 & 1 & 1 & 1\\
1 & 1 & 1 & 1
\end{array}
\right]  .
\]
We have
\[
\mathcal{E}(A)=\mathcal{\varepsilon}_{23}(A)\cup\mathcal{\varepsilon}%
_{24}(A)\cup\mathcal{\varepsilon}_{34}(A),
\]
with%
\begin{align*}
\mathcal{\varepsilon}_{23}(A)  &  =\left\{  w\in\mathbb{R}_{+}^{4}:3w_{1}\geq
w_{2}\geq w_{4}\geq w_{3}\geq2w_{1}\right\}  ,\\
\mathcal{\varepsilon}_{24}(A)  &  =\left\{  w\in\mathbb{R}_{+}^{4}:3w_{1}\geq
w_{2}\geq w_{3}\geq w_{4}\geq w_{1}\right\}  ,\\
\mathcal{\varepsilon}_{34}(A)  &  =\left\{  w\in\mathbb{R}_{+}^{4}:2w_{1}\geq
w_{3}\geq w_{2}\geq w_{4}\geq w_{1}\right\}  .
\end{align*}
We have%
\begin{align*}
u  &  =\left[
\begin{array}
[c]{cccc}%
1 & 3 & 3 & 3
\end{array}
\right]  ^{T}\in\mathcal{\varepsilon}_{23}(A),\\
v  &  =\left[
\begin{array}
[c]{cccc}%
1 & 1 & 2 & 1
\end{array}
\right]  ^{T}\in\mathcal{\varepsilon}_{34}(A).
\end{align*}
However,%
\[
u+v=\left[
\begin{array}
[c]{cccc}%
2 & 4 & 5 & 4
\end{array}
\right]  ^{T}\notin\mathcal{E}(A).
\]

\end{example}

\section{Conclusions\label{sconclusions}}

Examples of reciprocal matrices for which the set of efficient vectors is not
convex are known, for example those for which the right Perron vector is not
efficient. Here we have described the set of efficient vectors for a
reciprocal matrix $A$ as a union of at most $\frac{(n-1)!}{2}$ convex sets.
Each of these sets corresponds to a Hamiltonian cycle in $A$ whose product of
the entries is less than $1$ (when the matrix is inconsistent). Our
characterization allowed us to describe the set of efficient vectors for a
reciprocal matrix obtained from a consistent matrix by modifying one column
(row) as a union of at most $\frac{(n-1)(n-2)}{2}$ convex sets. In particular,
the known characterization of the efficient vectors for a reciprocal matrix
obtained from a consistent one by modifying one pair of reciprocal entries
followed in an easy way. In this case the set of efficient vectors is convex.
We also identified efficient vectors for $A$ with at most one order reversal
at positions along the associated cycle.

The results obtained here may be helpful in a better understanding of the
convexity of the set of efficient vectors for a reciprocal matrix, as well as
of the existence of rank reversals in these vectors.

\begin{acknowledgement}
We thank the referees for the helpful comments.
\end{acknowledgement}

\end{document}